\documentclass[10pt]{amsart}
\usepackage{latexsym}
\usepackage{stmaryrd}
\usepackage{amsmath}
\usepackage{epsfig}
\usepackage{graphicx}
\usepackage{eufrak}
\usepackage{dsfont}
\usepackage[english]{babel}
\usepackage{color}
\usepackage{mathabx}

\usepackage{amscd}      
\usepackage{amssymb}
\usepackage{dsfont}
\usepackage{xypic}      
\LaTeXdiagrams          
\usepackage[all,v2]{xy}
\xyoption{2cell} \UseAllTwocells \xyoption{frame} \CompileMatrices
\usepackage{latexsym}
\usepackage{epsfig}
\usepackage{enumerate}

\usepackage{latexsym}
\usepackage{epsfig}
\usepackage{amsfonts}
\usepackage{enumerate}
\usepackage{mathrsfs}

\allowdisplaybreaks[3]

\newtheorem{prop}{Proposition}[section]
\newtheorem{lem}[prop]{Lemma}

\newtheorem{cor}[prop]{Corollary}
\newtheorem{thm}[prop]{Theorem}
\newtheorem{rmk}[prop]{Remark}
\newtheorem{example}{Example}

\newtheorem{defn}[prop]{Definition}

            {\nolinebreak $\Box$ \end{trivlist}}

\newcommand{\noprint}[1]{}

\newcommand{\Hom}{\mbox{Hom}}

\newcommand{\et}{{\mbox{\tiny \'{e}t}}}

\newcommand{\tw}{\mbox{\tiny tw}}

\newcommand{\pt}{\mathop{pt}}

\newcommand{\E}{\mathop{\sf E}\nolimits}

\newcommand{\N}{\mathcal{N}}

\newcommand{\XX}{{\mathfrak X}}

\renewcommand{\SS}{{\mathfrak S}}

\newcommand{\YY}{{\mathfrak Y}}

\newcommand{\ZZ}{{\mathfrak Z}}

\newcommand{\Tt}{{\mathfrak t}}

\newcommand{\zz}{{\mathbb Z}}

\renewcommand{\ll}{{\mathbb L}}
\newcommand{\qq}{{\mathbb Q}}

\newcommand{\Gm}{{{\mathbb G}_{\mbox{\tiny\rm m}}}}

\newcommand{\sB}{{\mathcal B}}

\newcommand{\sE}{{\mathcal E}}

\newcommand{\sL}{{\mathcal L}}

\newcommand{\sO}{{\mathcal O}}

\newcommand{\sM}{{\mathcal M}}

\newcommand{\Coh}{\mbox{Coh}}

\newcommand{\rE}{\mathscr{E}}

\newcommand{\rT}{\mathscr{T}}
\newcommand{\rH}{\mathscr{H}}
\newcommand{\rP}{\mathscr{P}}
\newcommand{\rHT}{\mathscr{HT}}
\newcommand{\rHP}{\mathscr{HP}}
\newcommand{\rA}{\mathscr{A}}
\newcommand{\rHA}{\mathscr{HA}}

\newcommand{\HGA}{\mathbf{HGA}}
\newcommand{\GA}{\mathbf{GA}}
\newcommand{\HA}{\mathbf{HA}}

\newcommand{\SU}{\text{SU}}
\newcommand{\SL}{\text{SL}}
\newcommand{\PGL}{\text{PGL}}

\newcommand{\cHom}{\mathscr{H}om}

\DeclareMathOperator{\id}{id}

\DeclareMathOperator{\Aut}{Aut}

\DeclareMathOperator{\ind}{ind}

\DeclareMathOperator{\VW}{VW}

\DeclareMathOperator{\vir}{vir}

\DeclareMathOperator{\Pic}{Pic}
\DeclareMathOperator{\vd}{vd}

\DeclareMathOperator{\Higg}{Higg}

\DeclareMathOperator{\vb}{vb}

\DeclareMathOperator{\Tot}{Tot}

\DeclareMathOperator{\per}{per}

\DeclareMathOperator{\pparf}{parf}
\DeclareMathOperator{\triv}{triv}

\newcommand{\rk}{\mathop{\rm rk}}

\newcommand{\tr}{\mathop{\rm tr}\nolimits}

\newcommand{\spec}{\mathop{\rm Spec}\nolimits}

\newcommand{\tor}{\mathop{\rm tor}\nolimits}

\numberwithin{equation}{subsection}

\newcommand {\mat}      [1] {\left(\begin{array}{#1}}
\newcommand {\rix}          {\end{array}\right)}

\title[Construction of  moduli stack of projective Higgs bundles]{On the construction of moduli stack of projective Higgs bundles over surfaces}
\author{Yunfeng Jiang}
\address{Department of Mathematics\\ University of Kansas\\ 405 Snow Hall 1460 Jayhawk Blvd\\Lawrence KS 66045 USA} 
\email{y.jiang@ku.edu}

\begin{document}
\sloppy \maketitle
\begin{abstract}
We generalize the construction of M. Lieblich for the compactification of the moduli stack of $\PGL_r$-bundles on algebraic spaces to  the  moduli stack of Tanaka-Thomas  $\PGL_r$-Higgs bundles on algebraic schemes.  The method we use is  the moduli stack of Higgs version of Azumaya algebras. In the case of smooth surfaces, we obtain a virtual fundamental class on the moduli stack  of  $\PGL_r$-Higgs bundles. An application to the Vafa-Witten invariants is discussed. 
\end{abstract}

\maketitle

\tableofcontents

\section{Introduction}

It is well-known that the Langlands dual group of the group 
$\SL_r$ is  $\PGL_r$.  The moduli space $M_{r, L}(X)$ of slope semistable vector bundles on a smooth algebraic curve 
$X$ with rank $r$ and fixed determinant $L\in\Pic(X)$ is a smooth projective scheme.  The moduli space $N_{r, L}(X)$ of semistable Higgs bundles 
$(E,\phi)$ on $X$ with $\det(E)=L$ is a smooth hyperk\"ahler manifold.  The structure group for such moduli spaces is 
$\SL_r$. Motivated by mirror symmetry and Langlands duality, it is interesting to study the moduli space of $\PGL_r$-bundles or $\PGL_r$-Higgs bundles.  The research along this direction has been studied for a long time.  Especially in \cite{HT}, Hausel-Thaddeus have made a conjecture that the moduli space of Higgs bundles for the group $\SL_r$ and the moduli space of  $\PGL_r$-Higgs bundles on $X$ form a SYZ mirror partner.  Hausel-Thaddeus have proved several cases, and recently this conjecture was proved in \cite{GWZ}.
In \cite{Yun}, Yun has studied a more refined correspondence between the moduli stack of  parabolic $\SL_r$-Higgs bundles on $X$ and the moduli stack of parabolic $\PGL_r$-Higgs bundles on $X$ via Hecke correspondence. 

Let $X$ be a smooth projective surface.  The moduli space $\N:=\N_{r,L, c_2}(X)$ of stable Higgs sheaves $(E,\phi)$ was studied in \cite{TT1},  where a Higgs sheaf $(E,\phi)$ means that $E$ is a torsion free coherent sheaf  of rank $r$, $\det(E)=L$, and second Chern class $c_2(E)=c_2$; and $\phi: E\to E\otimes K_X$ an $\sO_X$-linear map called the Higgs field.  This is not the case of Simpson Higgs bundles or sheaves, where Simpson takes the Higgs field as 
$\theta: E\to E\otimes \Omega_X$.  The moduli space $\N$ is isomorphic to the moduli space of stable torsion two dimensional 
sheaves on the total space $Y:=\Tot(K_X)$ supported on $X$.  
The structure group for the moduli space $\N$ is $\SL_r$. 
In the Langlands dual side,  the moduli stack $\GA_{\XX}^{s}(r)$ of stable $\PGL_r$-bundles and its compactification was constructed in \cite{Lieblich_ANT} using the generalized Azumaya algebras. 

The S-duality conjecture of Vafa-Witten \cite{VW} predicted that the Langlands dual gauge group $\SU(r)/\zz_r$-Vafa-Witten theory can be obtained from the S-transformation of the gauge group $\SU(r)$-Vafa-Witten theory. More details can be found in 
\cite{VW}, \cite{Jiang_ICCM} and \cite{Jiang_2019}. The Vafa-Witten invariants for the surface $S$ and gauge group $\SU(r)$ are defined by Tanaka-Thomas \cite{TT1} using the moduli space $\N$ of  stable Higgs sheaves $(E,\phi)$.  
The method to define the invariants was studied in \cite{JT}, \cite{Jiang}. 
Tanaka-Thomas proved that there is a symmetric obstruction theory on $\N$ in the sense of Behrend \cite{Behrend}, and defined the Vafa-Witten invariants using virtual localization techniques in \cite{GP} on the moduli spaces.  They proved several cases of the Vafa-Witten predictions for the group $\SU(r)$.   To study the full S-duality conjecture in \cite{VW}, in \cite{Jiang_2019} the author has developed a theory of twisted Vafa-Witten invariants using the moduli space $\N^{\tw}:=\N^{\tw,s}_{\XX}(r,L,c_2)$ of stable $\XX$-twisted Higgs sheaves $(E,\phi)$ on $\XX$ with $\det(E)=L\in\Pic(\XX)$ and $c_2(E)=c_2$, where $\XX\to X$ is a $\mu_r$-gerbe over the surface $X$.  
The twisted invariants are also defined by virtual localization on the symmetric obstruction theory on $\N^{\tw}$.
In \cite{Jiang_2019}, the author conjectured that the twisted Vafa-Witten invariants of all the $\mu_r$-gerbes on $X$ will give the gauge group $\SU(r)/\zz_r$-Vafa-Witten invariants and satisfy the S-duality conjecture. 

In \cite{Lieblich_ANT}, for a $\mu_r$-gerbe $\XX\to X$ such that its order 
$|[\XX]|$ in the cohomological Brauer group $H^2(X,\Gm)_{\tor}$ is $r$ (which is called an optimal gerbe), 
Lieblich showed that the moduli stack $\sM^{\tw}$ of stable $\XX$-twisted sheaves on $\XX$ is a cover over the moduli stack 
 $\GA_{\XX}^{s}(r)$ of stable $\PGL_r$-torsors. 
We generalize the construction to the moduli stack $\HGA_{\XX}^{s}(r)$ of stable Higgs-Azumaya algebras of rank $r^2$. 
We show that there exists a covering morphism from the moduli stack $\N^{\tw}$ of 
stable $\XX$-twisted  Higgs sheaves to the  moduli stack $\HGA_{\XX}^{s}(r)$ of stable Higgs-Azumaya algebras of rank $r^2$, see Proposition \ref{prop_covering_morphism}. 
We then show that this morphism is compatible with the perfect obstruction theories on $\N^{\tw}$ and  $\HGA_{\XX}^{s}(r)$
respectively.  Thus the twisted  Vafa-Witten invariants $\VW^{\tw}$ defined by $\N^{\tw}$ can be obtained from the $\PGL_r$-invariants $\VW^{\PGL}$ defined by $\HGA_{\XX}^{s}(r)$, see Theorem \ref{thm_pullback_virtual_class}.
This provides  an evidence that the twisted Vafa-Witten invariants studied in \cite{Jiang_2019} are really the Langlands dual side 
$\SU(r)/\zz_r$-Vafa-Witten invariants due to the fact that the Langlands dual group of the structure group $\SL_r$ is $\PGL_r$.

\subsection{Outline}

This short note is organized as follows. In \S \ref{sec_preliminaries} we collect some results of twisted objects and rigidifications for a stack;  In \S \ref{sec_moduli_projective_Higgs_bundles} we provide the construction of the moduli stack of 
$\PGL_r$-Higgs bundles using the generalized Higgs-Azumaya algebras.  Finally in \S \ref{sec_POT_virtual_class_VW_Invariants} we talk about stability,  the construction of virtual fundamental classes and application to the Vafa-Witten invariants. 

\subsection{Convention}
We work over an algebraically closed field $\kappa$ of any character throughout of the paper. We denote by $\Gm$ the multiplicative group over $\kappa$.   We use Roman letter $E$ to represent a coherent sheaf on a projective DM stack or an \'etale gerbe  $\XX$, and use curl latter $\sE$ to represent the sheaves on the total space Tot$(\sL)$ of a line bundle $\sL$ over $\XX$. 
We keep the convention in \cite{VW} to use $\SU(r)/\zz_r$ as the Langlands dual group of $\SU(r)$. We always write 
$\SL_r:=\SL_r(\kappa)$ and $\PGL_r:=\PGL_r(\kappa)$.

\subsection*{Acknowledgments}

Y. J. would like to thank  Huai-Liang Chang, Amin Gholampour, Martijn Kool,  Richard Thomas and Hsian-Hua Tseng for valuable discussions.  This work is partially supported by  NSF DMS-1600997.


\section{Preliminaries on twisted objects and rigidifications}\label{sec_preliminaries}

\subsection{Notations}\label{subsec_notations}

We mainly follow the notation and construction of Lieblich in his beautiful paper \cite[\S3]{Lieblich_ANT}.  Let 
$X$ be a smooth projective scheme. (Note that Lieblich constructed the stack for any algebraic space $X$, for our purpose we fix to the case of schemes.) For the definition of stacks of sheaves, we use \cite{Stack_Project}, \cite{LMB}.

Let $E$ be a quasi-coherent sheaf of finite presentation on $X$. The definition of perfectness, pureness, total supportedness and total pureness of $E$ are given in \cite[\S3]{Lieblich_ANT}.  We use the following notations:
\begin{enumerate}
\item $\rT(X)$: the stack of totally supported sheaves;
\item  $\rT(X)^{\pparf}$: the stack of perfect totally supported sheaves;
\item $\rP(X)$: the stack of pure sheaves;
\item $\rP(X)^{\pparf}$: the stack of perfect pure sheaves.
\end{enumerate}

We need to define the corresponding stack of Higgs sheaves.

\begin{defn}\label{defn_Higgs_sheaf}
A Tanaka-Thomas Higgs sheaf on $X$ is a pair $(E,\phi)$, where $E\in \Coh(X)$ is a coherent sheaf, and 
$$\phi: E\to E\otimes K_{X}$$
is a $\sO_X$-linear map which is called the Higgs field. 
\end{defn}
\begin{rmk}
Let $Y:=\Tot(K_X)$ be the total space of the canonical line bundle $K_X$. 
According to \cite[Proposition 2.6]{TT1}, the category of Higgs sheaves on $X$ is equivalent to the category of 
$\dim(X)$-dimensional torsion sheaves on $Y$ supported on the zero section $X\subset \Tot(K_X)$. 
\end{rmk}
We define the corresponding stacks of Higgs sheaves. 
\begin{enumerate}
\item $\rHT(X)$: the stack of totally supported Higgs sheaves;
\item  $\rHT(X)^{\pparf}$: the stack of perfect totally supported Higgs sheaves;
\item $\rHP(X)$: the stack of pure Higgs sheaves;
\item $\rHP(X)^{\pparf}$: the stack of perfect pure Higgs sheaves.
\end{enumerate}
\begin{rmk}
A Higgs sheaf $(E,\phi)$ on $X$ is called ``totally supported", ``perfect totally supported", ``pure" and 
``perfect pure" if the corresponding sheaf $E$ is  ``totally supported", ``perfect totally supported", ``pure" and 
``perfect pure" respectively. 
\end{rmk}

\subsection{Construction of Lieblich}\label{subsec_Lieblich}

In \cite[\S 3.4]{Lieblich_ANT}, Lieblich constructed the ``twisted objects and rigidification" for a stack.  We recall it here. 
All the basic knowledge, like ``the relative small \'etale site" on a scheme $X$, see \cite[\S 3]{Lieblich_ANT}.

Let $A$ be an abelian group scheme on $X$, and $\XX\to X$ an $A$-gerbe over $X$. Then we have 
$A_{\XX}\cong I\XX$, where $A_{\XX}$ is the sheaf of abelian group $A$ on $\XX$ and $I\XX$ is the inertia stack of $\XX$. 
In the following let us fix a stack $\ZZ\to X$ over the scheme $X$, and let 
$$\chi: A\to I\ZZ$$
be a central injection into the inertia stack of $\ZZ$.

\begin{defn}\label{defn_twisted_section_ZZ}
For the stack $\ZZ$, an $\XX$-twisted section of $\ZZ$ over a scheme $T\to X$ is a $1$-morphism $f: \XX\times_{X}T\to \ZZ$ such that the induced map $A\to I(\XX\times_{X}T)\to f^*(I\ZZ)$ is identified with the pullback under $f$ of the canonical inclusion $\chi: A\to I\ZZ$.

We denote by $\ZZ^{\XX}$ the substack of the Hom stack $\Hom_{X}(\XX,\ZZ)$ consisting of $\XX$-twisted sections of 
$\ZZ$.
\end{defn}

\subsubsection{Cocycle twisted objects}\label{subsubsec_cocycle}

The $A$-gerbes on $X$ are classified by the second \'etale cohomology group $H^2_{\et}(X,A)$. Then we can represent the $A$-gerbe $\XX\to X$ as a cohomology class $\alpha_{\XX}\in H^2_{\et}(X,A)$.  Let 
$\{U_{\bullet}\}\to X$ be a hypercovering such that 
\begin{enumerate}
\item there exists a section $\sigma$ of $\XX$ on $U_0$, and \\
\item the two pullbacks of $\sigma$ to $U_1$ are isomorphic via $\phi$.
\end{enumerate}
The coboundary of $\phi$ gives the cocycle $\alpha_{\XX}$. In general, this $A$-gerbe is the stack of ``twisted $A$-torsors". 

\begin{defn}\label{defn_cocycle_twisted_sections}
Given a class $\alpha_{\XX}\in H^2_{\et}(X,A)$, we fix a hypercover $\{U_{\bullet}\}$ and write the cocycle as 
$(\alpha_{\XX}, U_{\bullet})$.  A $(\alpha_{\XX}, U_{\bullet})$-twisted section of the stack $\ZZ$ over 
$T\to X$ is given by 

$\bullet$  a $1$-morphism $\phi: U_0\times_{X}T\to \ZZ$, and 

$\bullet$ a $2$-morphism $\psi: (p_0^\prime)^*\phi\stackrel{\sim}{\longrightarrow}(p_1^\prime)^*\psi$,
where $p_0^\prime$ and $p_1^\prime$ are the two natural maps $U_1\times_{X}T\to U_0\times_{X}T$, subject to 
the condition that the coboundary $\delta\psi\in \Aut((p_0^\prime)^*\phi)$ is equal to the action of $\alpha_{\XX}$.
\end{defn}

\begin{prop}(\cite[Proposition 3.4.1.7]{Lieblich_ANT})
There is a natural equivalence between the stack of $\XX$-twisted objects of $\ZZ$ and 
$(\alpha_{\XX}, U_{\bullet})$-twisted objects of $\ZZ$.
\end{prop}

\subsubsection{Pushing forward rigidifications}\label{subsubsec_pushing}

We recall the pushing forward and rigidification construction of Lieblich and Abramovich-Corti-Vistoli \cite{ACV}. 
Let $\XX\to X$ be an $A$-gerbe and $\ZZ$ an $X$-stack.  Assume that there is a central injection:
$$A_{\ZZ}\hookrightarrow I\ZZ,$$
then  we denote by 
$\ZZ \fatslash A$ the ``rigidification" of $\ZZ$ along $A$. 
Let $f: X\to S$ be a smooth morphism of schemes. From \cite[\S 3.4.2.]{Lieblich_ANT}, there is a morphism
$$f_*(\ZZ)\to f_*(\ZZ\fatslash A).$$
This morphism implies that given a morphism 
$$T\to S$$
of schemes, and a $1$-morphism $g: T\to f_*(\ZZ\fatslash A)$, there is an $A_T$-gerbe on $X\times_{T}S$ which is given by the morphism $\ZZ\to \ZZ\fatslash A$ (which is an $A$-gerbe) and the fact that $T\to f_*(\ZZ\fatslash A)$ corresponds to a morphism $X\times_{S}T\to \ZZ\fatslash A$. 

Let us recall the  gerbe $\XX$-twisted part of  $f_*(\ZZ\fatslash A)$.

\begin{defn}\label{defn_A_gerbe_twisted_ZZ}(Lieblich)
Let $\XX\to A$ be an $A$-gerbe.  We use  $f_*^{\XX}(\ZZ\fatslash A)$ to denote by the stack-theoretic image of  $f_*(\ZZ^{\XX})$ under the natural morphism 
$f_*(\ZZ^{\XX})\hookrightarrow f_*(\ZZ)\to f_*(\ZZ\fatslash A)$. The part $f_*^{\XX}(\ZZ\fatslash A)$ is called the ``$\XX$-twisted" part of $f_*(\ZZ\fatslash A)$.
\end{defn}

\subsubsection{Deformation theory}\label{subsubsec_Deformation}

We also recall the deformation theory results in  \cite[\S 3.4.3.]{Lieblich_ANT}. For our purpose, we always assume $f: X\to S$ is a smooth morphism between smooth schemes.  In particular we will take $S=\spec(\kappa)$. 

Let $A$ be an abelian group scheme over $X$.  Then from  \cite[Lemma 3.4.3.1]{Lieblich_ANT}, the natural morphism 
$$f_*\sB A\to R^1f_*(A)$$
is a $f_*A$-gerbe over $R^1f_*(A)$, and the cotangent complex of $f_*\sB A$ is trivial. We list a result of Lieblich here for a later use. 

\begin{prop}\label{prop_Lieblich_section3}(\cite[Corollary 3.4.3.3, Proposition 3.4.3.5]{Lieblich_ANT})
The natural map
$$\gamma: f_*(\ZZ^{\XX})\fatslash f_*(A)\to f_*^{\XX}(\ZZ\fatslash A)$$
is representable by finite \'etale covers.  

Also the stack $f_*(\ZZ^{\XX})$ is an Artin (resp. DM) stack if and only if the stack $f_*^{\XX}(\ZZ\fatslash A)$ is an Artin 
(resp. DM) stack. 
\end{prop}

The result in Proposition \ref{prop_Lieblich_section3} implies the relationship of virtual fundamental classes on 
$f_*(\ZZ^{\XX})$ and $f_*^{\XX}(\ZZ\fatslash A)$ if they admit perfect obstruction theories, see Proposition 3.4.3.7, and Corollary 3.4.3.8. in \cite{Lieblich_ANT}.

\section{Moduli stack construction of projective   Higgs  bundles}\label{sec_moduli_projective_Higgs_bundles}

In this section we present two approaches for the construction of the moduli stack compactification of $\PGL_r$-Higgs bundles using the result in \S \ref{sec_preliminaries}. In \cite[\S 4, \S 5]{Lieblich_ANT}, Lieblich constructed some compactification of the moduli of $\PGL_r$-torsors and his method works for the stack of Higgs bundles. 

\subsection{Abstract approach}\label{subsec_abstract_approach}

Let us fix a smooth proper morphism $f: X\to S$, and $p: \XX\to X$ a cyclic group $\mu_r$-gerbe over $X$. 
The map $f$ also represents the geometric morphism $X_{r\et}\to S_{\et}$ as in \cite[\S 2]{Lieblich_ANT}. We recall the stack of twisted Higgs sheaves in \cite{Jiang_2019}.

\subsubsection{Gerbe twisted sheaves}\label{subsubsec_gerbe_twisted_sheaf}

Let $\pi: \XX\to X$ be a $\mu_r$-gerbe and we denote by $[\XX]$ to be the class in $H^2(X,\mu_r)$.  Let $E$ be a   sheaf on $\XX$, then  there is a natural right group action
$$\mu: E\times I\XX\to E,$$
see \cite[Lemma 2.1.1.8]{Lieblich_Duke}. 
 Let $\chi: \mu_r\to \Gm$ be the inclusion character morphism.  Let $E$ be a coherent $\sO_{\XX}$-module, the module action $m: \Gm\times E\to E$ yields an associated right action $m^\prime: E\times\Gm\to E$ by 
$m^\prime(s,\varphi)=m(\varphi^{-1},s)$.

\begin{defn}\label{defn_gerbe_twisted_sheaf}
A $\chi$-twisted sheaf or $\XX$-twisted sheaf on $\XX$ is a  coherent $\sO_{\XX}$-module $E$ such that the natural action $\mu: E\times \mu_r\to E$ given by the $\mu_r$-gerbe structure
makes the diagram
\[
\xymatrix{
E\times \mu_r\ar[r]\ar[d]_{\chi}& E\ar[d]^{\id}\\
E\times \Gm\ar[r]^{m^\prime}& E
}
\]
commutes. 
\end{defn}

\begin{example}\label{example_Gm_gerbe_twist}
We also can define the $\Gm$-gerbe $\XX\to X$ twisted sheaf.  Then let $\{U_i\}$ be an open covering of $X$.  A  $\XX$-twisted sheaf on $X$ is given by:

$\bullet$ a sheaf of modules $E_i$ on each $U_i$;

$\bullet$ for each $i$ and $j$ an isomorphism of modules 
$g_{ij}: E_j|_{U_{ij}}\stackrel{\cong}{\longrightarrow}E_i|_{U_{ij}}$ such that on $U_{ijk}$, 
$E_k|_{U_{ijk}}\stackrel{\cong}{\longrightarrow}E_k|_{U_{ijk}}$ is equal to the multiplication by the scalar
$a\in \Gm(U_{ijk})$ giving the $2$-cocycle $[\XX]\in H^2(X,\Gm)$. 
\end{example}

\subsubsection{Twisted Higgs sheaves}\label{subsubsec_twisted_Higgs_sheaf}

We fix a polarization $\sO_{X}(1)$ on $X$ for the $\mu_r$-gerbe $\XX\to X$.  

Let us denote by $\YY:=\Tot(K_{\XX})$, the canonical line bundle of $K_{\XX}$.  This $\YY$ is a Calabi-Yau  stack. 
Here we let $\Coh_{c}(\Tot(K_{\XX}))$ be the abelian category of compactly supported sheaves on $\Tot(K_{\XX})$. We fix the following diagram:
\begin{equation}\label{diagram_XX_YY}
\xymatrix{
\YY\ar[r]^{p}\ar[d]_{\pi}& Y\ar[d]^{\pi}\\
\XX\ar[r]^{p}& X
}
\end{equation}
where $Y=\Tot(K_X)$ is the total space of $K_X$, which is the coarse moduli space of $\YY$. 
It is not hard to see that 
the DM stack $\YY\to Y$ is also a $\mu_r$-gerbe, and the class 
$[\YY]\in H^2(Y,\mu_r)\cong H^2(X,\mu_r)$. 

\begin{defn}\label{defn_twisted_Higgs_pair_XX}(\cite[Definition 3.32]{Jiang_2019})
An $\XX$-twisted Higgs sheaf on $\XX$ is a pair $(E,\phi)$, where $E$ is a $\XX$-twisted  coherent sheaf on $\XX$ as in Definition \ref{defn_gerbe_twisted_sheaf}, and $\phi: E\to E\otimes K_{\XX}$ a $\sO_{\XX}$-linear morphism such that the following diagram 
\[
\xymatrix{
E\times \mu_r\ar[r]\ar[d]_{\chi}& E\ar[d]^{\id}\ar[r]^{\phi}& E\otimes K_{\XX}\ar[d]^{\id}\\
E\times \Gm\ar[r]^{m^\prime}& E\ar[r]^{\phi}& E\otimes K_{\XX}
}
\]
commutes.
\end{defn}
From  \cite[Proposition 2.18]{JP} and \cite[Proposition 2.2]{TT1}, for a $\mu_r$-gerbe $\XX\to X$, 
there exists an abelian category $\Higg^{\tw}_{K_{\XX}}(\XX)$ of $\XX$-twisted Higgs pairs on $\XX$ and an equivalence:
\begin{equation}\label{eqn_equivalence_twisted_categories}
\Higg^{\tw}_{K_{\XX}}(\XX)\stackrel{\sim}{\longrightarrow} \Coh^{\tw}_{c}(\YY)
\end{equation}
where $\Coh^{\tw}_{c}(\YY)$ is the category of compactly supported $\YY$-twisted coherent sheaves on $\YY$. 

As in \cite[\S 3.4]{Jiang_2019}, we let 
$\N^{\tw}_{\XX/X}(r,\sO)$ be the moduli stack of $\XX$-twisted Higgs sheaves $(E,\phi)$ with rank $r$ and 
$\det(E)=\sO$.  Note that we haven't put any stability. 
Here is a generalization of Lemma 4.2.2. in \cite{Lieblich_ANT} to Higgs sheaves:
\begin{prop}\label{prop_Higgs_stack_locally_finite}
Let $\rHT^{\sO}_{X/S}(r)$ be the stack of totally supported Higgs sheaves with fixed determinant $\sO$.  Then the stack 
$f_*\left(\rHT^{\sO}_{X/S}(r)\fatslash \mu_r\right)$ is an Artin stack locally of finite presentation over $S$.  If the morphism $f$ is smooth, then the stack is quasi-proper. 
\end{prop}
\begin{proof}
The proof is the same as  in \cite[Lemma 4.2.2.]{Lieblich_ANT}. It is sufficient to show locally (hence globally) the morphism 
$\gamma$ in Proposition \ref{prop_Lieblich_section3} makes $f_*\left(\rHT^{\sO}_{X/S}(r)\fatslash \mu_r\right)$ a $R^2f_*(\mu_r)$-stack of finite presentations.  Hence from the cohomology class of the $\mu_r$-gerbe given by a section of $R^2f_*(\mu_r)$, it is sufficient to show that given a $\mu_r$-gerbe $\XX\to X$, the stack $f^{\XX}_*\left(\rHT^{\sO}_{X/S}(r)\fatslash \mu_r\right)$  is an Artin stack locally of finite presentation. 

From Proposition \ref{prop_Lieblich_section3}, it is sufficient to show that $f_*\left(\left(\rHT^{\sO}_{X/S}(r)\right)^{\XX}\right)$ is an Artin stack locally of finite presentation.  This is true since it is an open substack of the stack of perfect $\XX$-twisted Higgs sheaves with trivial determinant. 

The proof of quasi-properness is similar with  \cite[Lemma 4.2.2.]{Lieblich_ANT}, and is reduced to prove $f_*\left(\left(\rHT^{\sO}_{X/S}(r)\right)^{\XX}\right)$  is quasi-proper.  Then the proof is reduced to show that given a discrete valuation ring $R$ and a torsion free $\XX$-twisted Higgs sheaf $(E,\phi)$ of rank $r$ with $\det(E)=\sO$ over the fractional field $F$ of $R$, there is an extension of $E$ to a flat family over a finite flat extension of $R$ such that the condition 
$\det(E)=\sO$ extends to the whole $R$. 
From   \cite[Lemma 4.2.2.]{Lieblich_ANT}, the coherent sheaf $E$ extends. It only need to check that the Higgs field 
$\phi: E\to E\otimes K_{\XX}$ extends too.  But this is from $\phi$ is $\sO_{\XX}$-linear.
\end{proof}

We then define the following:

\begin{defn}\label{defn_Abstract_Higgs_PGL}
Let $\XX\to X$ be a $\mu_r$-gerbe over the scheme 
$X$. We define
$$\rH^{\XX}_{r}:=f^{\XX}_*\left(\rHP^{\sO}_{X/S}(r)\fatslash \mu_r\right).$$
This is the stack of $\PGL_r$-Higgs torsors.
\end{defn}

\subsection{The construction using generalized Azumaya algebras}\label{subsec_construction_Azumaya}

We give the construction of generalized Higgs version of Azumaya algebras.  Lieblich used the derived Skolem-Noether in the derived category $D(\sO_X)$ of $\sO_X$-modules to study generalized Azumaya algebras. For the basic facts about the weak algebras, see \cite[\S 5.1]{Lieblich_ANT}.

\begin{defn}\label{defn_generalized_Azumaya}
A pre-generalized Azumaya algebra on $X$ is a perfect algebra $A$ in the derived category $D(\sO_X)$ of $\sO_X$-modules such that there exists an open cover $\{U_i\}\in X_{\et}$ and a totally supported perfect sheaf $E$ on $\{U_i\}$ such that 
$$A|_{U_i}\cong R\sE nd_{U_i}(E)$$
as weak algebras.  An isomorphism of pre-generalized Azumaya algebras is an isomorphism in the category of weak algebras. 
\end{defn}

\begin{defn}\label{defn_generalized_Higgs_Azumaya}
A pre-generalized  Higgs-Azumaya algebra on $X$ is given by a pair $(A,\Phi)$, where $A$ is a pre-generalized Azumaya algebra on $X$, and $\Phi$ is the Higgs section such that locally on $U$
we have a diagram
\[
\xymatrix{
A|_{U}\ar[r]^{\cong}\ar[d]& R\sE nd_{U}(E)\ar[d]\\
A\otimes K_{X}|_{U}\ar[r]^{\cong}&  R\sE nd_{U}(E\otimes K_{X})
}
\]
where $A\otimes K_X$ exists as a weak algebra. 
\end{defn}

\subsubsection{Lieblich's stack of generalized Azumaya algebras}

Using Definition \ref{defn_generalized_Azumaya} Lieblich defined the stack $\rA_{X}$ of generalized Azumaya algebras on $X$.  Let us recall the degree of the generalized Azumaya algebra $A$ on $X$.  
Associated to $A$, there is a $\Gm$-gerbe $\XX(A)\to X$, called  the gerbe of trivializations of $A$, is by definition the stack on the small site $X_{\et}$ such that its sections over $U\to X$ is given by pairs $(E,\varphi)$, where $E$ is a totally supported sheaf on $U$ and 
$\varphi: R\sE nd_U(E)\stackrel{\sim}{\longrightarrow} A|_{U}$ an isomorphism of generalized Azumaya algebras. 

\begin{defn}\label{defn_degree_Azumaya_algebra}
The generalized  Azumaya algebra $A$ on $X$ is of degree $r$ if for any $U\to X$, 
$$A|_U\cong R\sE nd_U(E)$$
and $E$ has rank $r$ such that $\rk(A)=r^2$. 
The degree of a generalized  Higgs-Azumaya algebra  $(A,\Phi)$ is the degree of $A$.
\end{defn}

Let $A$ be a generalized Azumaya algebra of degree $r$ on $X$.  We define the ``gerbe of trivialized trivialization" of $A$, which is denoted by $\XX_{\triv}(A)$, to be  the  stack on the small \'etale site of $X$ such that its sections over $U\to X$ are given by triples $(E,\varphi, \delta)$, where on an open cover $U\to X$, 
$$\varphi: R\sE nd_U(E)\cong A|_{U}$$
is the isomorphism of generalized Azumaya algebras and 
$\delta: \sO_U\stackrel{\sim}{\longrightarrow} \det(E)$ is the isomorphism of invertible sheaves  after taking determinant. 
The isomorphisms in the fiber categories are isomorphisms of the sheaves which preserve the identifications with $A$ and the trivializations of the determinants. 

\begin{rmk}
\begin{enumerate}
\item From \cite[Lemma 5.2.1.8]{Lieblich_ANT}, the gerbe $\XX(A)\to X$ is a $\Gm$-gerbe on $X$; and its cohomology class 
is $[\XX(A)]\in H^2(X,\Gm)$. 
\item The $\mu_r$-gerbe $\XX_{\triv}(A)\to X$ determines a cohomology class $[\XX_{\triv}(A)]\in H^2(X,\mu_r)$ which maps to $[\XX(A)]\in H^2(X,\Gm)$ from the exact sequence
$$\cdots H^1(X,\Gm)\rightarrow H^2(X,\mu_r)\rightarrow H^2(X,\Gm)\cdots $$
which is given by the short exact sequence:
$$1\to \mu_r\longrightarrow \Gm\stackrel{(\cdot)^r}\longrightarrow \Gm\to 1.$$
\end{enumerate}
\end{rmk}

\subsubsection{Stack of generalized Higgs-Azumaya algebras}

\begin{defn}\label{defn_Higgs_Azumaya}
Define $\rHA_{X}(r)$ to be the stack of generalized Higgs-Azumaya algebras on $X$ of degree $r$. 
\end{defn}

We generalize Proposition 5.2.2.1 and Proposition 5.2.2.3 in \cite{Lieblich_ANT}.

\begin{prop}\label{prop_Higgs_Azumaya_morphisms}
The morphism 
$$\gamma: \rHT_{X}^{\pparf}(r)\to \rHA_X(r)$$
given by
$$(E,\phi)\mapsto (R\sE nd(E), \Phi)$$
gives rise to an isomorphism 
$$\rHT_X(r)\fatslash \Gm\stackrel{\sim}{\longrightarrow}\rHA_X(r).$$
And the morphism 
$$\gamma^{\sO}: \rHT_X^{\sO}(r)\to \rHA_X(r)$$
given by 
$$(E,\phi)\mapsto (R\sE nd(E), \Phi)$$
yields an isomorphism $\rHT_X^{\sO}(r)\fatslash \mu_r\stackrel{\sim}{\longrightarrow}\rHA_X(r)$.
\end{prop}
\begin{proof}
The proof for the generalized Azumaya algebra $A$ and the corresponding sheaf $E$ are given in 
\cite[Proposition 5.2.2.1, Proposition 5.2.2.3]{Lieblich_ANT}. The only thing to check is that the morphisms are compatible with the Higgs fields $\phi$ and $\Phi$. 
\end{proof}

\begin{rmk}
Forgetting about the Higgs fields or setting $\phi=0, \Phi=0$, we get the morphisms 
$$\gamma: \rT_X(r)\to \rA_X(r)$$
and 
$$\gamma^{\sO}: \rT_X^{\sO}(r)\to \rA_X(r)$$
in \cite{Lieblich_ANT}. 
\end{rmk}

Now let $f: X\to S$ be a smooth morphism and we define the relative generalized Azumaya algebras. 
\begin{defn}\label{defn_Higgs_Azumaya_relative}
A relative generalized Higgs-Azumaya algebra on $X/S$ is a generalized Higgs-Azumaya algebra $(A,\Phi)$ on $X$
such that the local sheaves are $S$-flat and totally pure on each geometric fiber. 
\end{defn}

\begin{rmk}
The definition \ref{defn_Higgs_Azumaya_relative} is equivalent to the following: 
$A\cong R\pi_*R\sE nd_{\XX}(E)$ where 
$\XX\to X$ is a $\Gm$-gerbe and $E$ is a $S$-flat $\XX$-twisted sheaf which is totally 
pure on each geometric fiber. 
\end{rmk}

\begin{defn}\label{defn_HGA}
We define $\HGA_{\XX/S}(r)$ to be the stack of generalized Higgs-Azumaya algebras on $X/S$ of rank 
$r^2$ such that on each geometric fiber of $f: X\to S$ the class of the Azumaya algebra agrees with 
$[\XX]\in H^2(X,\mu_r)$ \'etale locally around every point in the base. 

We use the notation $\HGA_{X/S}(r)$ to represent the stack of generalized Higgs-Azumaya algebras on $X/S$ of rank 
$r^2$ on each fiber without mentioning the class $[\XX]$.
\end{defn}

\begin{prop}\label{prop_HGA_Hr}
There is an isomorphism of stacks:
$$\HGA_{\XX/S}(r)\stackrel{\sim}{\longrightarrow}\rH_r^{\XX}.$$
\end{prop}
\begin{proof}
Recall from Definition \ref{defn_Abstract_Higgs_PGL}, 
$$\rH^{\XX}_{r}=f^{\XX}_*\left(\rHP^{\sO}_{X/S}(r)\fatslash \mu_r\right).$$
The stack $\rHA_X(r)$ parametrizes the generalized Higgs-Azumaya algebras $(A,\Phi)$ on $X$ of degree 
$r$, such that locally they are isomorphic to $R\sE nd(E)$ for an object 
$E\in \rHP_{X/S}(r)$.  So when we fix the determinant,  the natural map:
$$\rHP_{X/S}^{\sO}(r)\to \rHA_X(r)$$
induces an isomorphism 
\begin{equation}\label{eqn_isomorphism_HP_HA}
\rHP_{X/S}^{\sO}(r)\fatslash \mu_r\stackrel{\sim}{\longrightarrow} \rHA_X(r).
\end{equation}
For the morphism $f: X\to S$, it is not hard to check that:
$$\HGA_{X/S}(r)\stackrel{\sim}{\longrightarrow} f_*(\rHA_X(r))$$
and when  fixing the cohomology class $[\XX]\in H^2(X,\mu_r)$ we have: 
$$\HGA_{\XX/S}(r)\stackrel{\sim}{\longrightarrow} f^{\XX}_*(\rHA_X(r))=\rH_r^{\XX}.$$
\end{proof}

\section{Stability, virtual fundamental classes and application to Vafa-Witten invariants}\label{sec_POT_virtual_class_VW_Invariants}

In this section we fix a smooth projective surface $f: X\to S=\spec(\kappa)$. We define stability for Higgs Azumaya algebras. As pointed out by Lieblich, it is not very convenient to define Gieseker stability for Higgs-Azumaya algebras and we fix to slope stability. 

\subsection{Stability conditions}\label{subsec_stability_condition}

Recall that for a torsion-free sheaf $E$ on $X$, the slope of $E$ is defined as:
$\mu(E)=\frac{\deg(E)}{\rk(E)}$. The Hilbert polynomial of $E$ is given by $H(E,m)=\chi(X, E\otimes \sO_X(m))$. For a coherent sheaf $E$ on the $\mu_r$-gerbe $\XX\to X$, we us the Vistoli Chow group 
$A_*(\XX)\cong A_*(X)$ with $\qq$-coefficient. There is a degree map
$$d: A_0(\XX)\to \qq$$
such that $d([pt])=\frac{1}{r}$ if $[\pt]$ is a $0$-cycle supported on a single point of $X$.  We define a normalized degree function 
$$\deg: A_0(\XX)\to \qq$$
by $\deg([pt])=r d([\pt])=1$.
The slope of $E$ on $\XX$ is then again defined as $\mu(E) =\frac{\deg(E)}{\rk(E)}$.

\begin{defn}(Slope stability)
A torsion-free sheaf $E$ on $\XX$ is stable if for any subsheaf $F\subset E$, $\mu(F)<\mu(E)$.
\end{defn}

Now let $(A,\Phi)$ be a Higgs-Azumaya algebra on $X$. The Higgs field is 
$\Phi: A\to A\otimes K_X$.  For the Azumaya algebra $A$ on $X$, we define 
$$\mu(A)=\frac{\deg(A)}{\rk(A)}$$
where $\deg(A)$ is defined as the degree of the adjoint sheaf $E$. 

\begin{defn}\label{defn_stability_Azumaya}(\cite[Definition 6.1.2]{Lieblich_ANT})
An Higgs-Azumaya algebra $(A,\Phi)$ on $X$ is called stable if for all  non-zero $\Phi$-invariant 
right ideals $I\subset A$ of rank strictly smaller than $\rk(A)$ we have $\mu(I)<0$.
\end{defn}

\begin{lem}
For a $\mu_r$-gerbe $p: \XX\to X$, a locally free $\XX$-twisted sheaf $E$ is stable if and only if the Azumaya algebra 
$p_*\sE nd(E)$ is stable. 
\end{lem}
\begin{proof}
This is from that any $\Phi$-invariant right ideal of $A$ has the form 
$\Hom(E,F)$ for a $\phi$-invariant subsheaf $F\subset E$. 
\end{proof}

We let $$\HGA^{s}_{\XX/S}(r)\subset \HGA_{\XX/S}(r)$$
denote the open substack of $\HGA_{\XX/S}(r)$ parametrizing stable generalized Higgs-Azumaya algebras by the isomorphism 
$\HGA_{\XX/S}(r)\cong \rH_r^{\XX}$ in Proposition \ref{prop_HGA_Hr}.  Therefore 
$\HGA^{s}_{\XX/S}(r)\cong \left(\rH_r^{\XX}\right)^s$.

\subsection{Obstruction theory for twisted sheaves}\label{subsec_POT_twisted_sheaves}

Let $f: X\to S$ be a smooth morphism, and 
let us fix a $\mu_r$-gerbe $\XX\to X$ such that $\per(\XX):=|[\XX]|=r$ in $H^2(X,\Gm)_{\tor}$, i.e., it is an $\mu_r$-optimal gerbe.  Then from de Jong's result \cite{de_Jong}, the index $\ind(\XX)=\per(\XX)$, where 
$\ind(\XX)$ is the minimal rank $r>0$ such that there exists a $\XX$-twisted locally free sheaf of rank $r$.  Moreover, any such a $\XX$-twisted sheaf $E$ is stable automatically. 

Let  $\sM^{s,\tw}:=\sM^{s,\tw}_{\XX/\kappa}(r, L, c_2)$ be the moduli stack  of stable $\XX$-twisted torsion free sheaves with rank $r$, fixed determinant 
$L\in\Pic(\XX)$ and second Chern class $c_2$.  Let 
$$\E\to \XX\times \sM^{s,\tw}_{\XX/\kappa}(r, L, c_2)$$
be the universal twisted sheaf on $\XX\times \sM^{s,\tw}$. Let 
$$p_{\XX}: \XX\times \sM^{s,\tw}_{\XX/\kappa}(r, L, c_2)\to \XX$$ and 
$$p_{\sM}: \XX\times \sM^{s,\tw}_{\XX/\kappa}(r, L, c_2)\to \sM^{s,\tw}_{\XX/\kappa}(r, L, c_2)$$
be the projections.  Then from  \cite[\S 6.5.1]{Lieblich_Compo}, 
we have 
$$\Phi: R p_{\sM*}\left(R\cHom(\E, \ll p_{\XX}^*\omega_{\XX}\otimes^{\\}\E\right))\longrightarrow \ll_{\sM^{s,\tw}_{\XX/\kappa}}[-1]$$
and we restrict it to the traceless part and get a morphism:
\begin{equation}\label{eqn_POT_sM}
\Phi: R p_{\sM*}\left(R\cHom(\E, \ll p_{\XX}^*\omega_{\XX}\otimes^{\\}\E\right))_0\longrightarrow \ll_{\sM^{s,\tw}_{\XX/\kappa}}[-1]. 
\end{equation}
Then the shift of $\Phi$ by $1$, $\Phi[1]$ is a perfect obstruction theory for the moduli stack $\sM^{s,\tw}_{\XX/\kappa}$ in the sense of \cite{BF}, \cite{LT}.

\subsection{Deformation theory and virtual classes for $\HGA^{s}_{\XX/S}$}\label{subsec_deformation_virtual_class}

Recall that $\N^{\tw}:=\N^{\tw,s}_{\XX/S}(r,\sO,c_2)$ is the moduli stack of $\XX$-twisted stable Higgs sheaves of rank $r$ with fixed determinant $\sO$ and second Chern class $c_2$.  If we let 
$\N_{\vb}^{\tw,s}:=\N^{\tw,\vb, s}_{\XX/S}(r,\sO,c_2)$ be the substack of $\N^{\tw,s}_{\XX/S}(r,\sO,c_2)$ parametrizing locally free $\XX$-twisted  rank $r$ sheaves, then $\N_{\vb}^{\tw,s}$ is schematically dense. 

Recall from \S \ref{subsec_notations}, the stack $\rHP_{X/S}^{\sO}(r)$ is the stack of  Higgs pure sheaves on $X$ of rank $r$
and $\det(E)=\sO$.  Let  $\left(\rHP_{X/S}^{\sO}(r)\right)^{\XX}$ be the substack of the $\XX$-twisted pure sheaves of 
$\rHP_{X/S}^{\sO}(r)$. Then 
$$f_*\left(\left(\rHP_{X/S}^{\sO}(r)\right)^{\XX}\right)=f_*^{\XX}\left(\rHP_{X/S}^{\sO}(r)\right).$$
Also if we let $\left(\rHP_{X/S}^{\sO}(r)\right)^{\XX,s}$ be the substack of $\left(\rHP_{X/S}^{\sO}(r)\right)^{\XX}$ consisting of stable $\XX$-twisted pure Higgs sheaves of rank $r$, $\det(E)=\sO$, then 
$$\left(\rHP_{X/S}^{\sO}(r)\right)^{\XX,s}\cong \N^{\tw, s}_{\XX/S}(r,\sO).$$
From (\ref{eqn_isomorphism_HP_HA}),  
$\rHP_{X/S}^{\sO}(r)\fatslash \mu_r\stackrel{\sim}{\longrightarrow} \rHA_X(r)$. After fixing the cohomology class of the gerbe 
$[\XX]$, we have:
$$\left(\rHP_{X/S}^{\sO}(r)\right)^{\XX,s}\fatslash \mu_r\stackrel{\sim}{\longrightarrow} \left(\rHA_X(r)\right)^{\XX,s},$$
where 
$\left(\rHA_X(r)\right)^{\XX,s}\subset \rHA_X(r)$ is the substack of stable generalized Higgs-Azumaya algebras 
$(A,\Phi)$ on $X$ with cohomology class $[\XX]\in H^2(X,\mu_r)$.  Hence there is a morphism:
\begin{equation}\label{eqn_morphism_Ntw_HA}
\Psi_1: \N^{\tw,s}_{\XX/S}(r,\sO)\to \left(\rHA_X(r)\right)^{\XX,s}
\end{equation}
which is a $\mu_r$-gerbe.  Also from the morphism $f: X\to \spec(\kappa)=S$, we have a morphism 
\begin{equation}\label{eqn_morphism_HA_HGA}
\Psi_2:  \left(\rHA_X(r)\right)^{\XX,s}\to f_*^{\XX,s}(\rHA_X(r))=(\rH_r^{\XX})^{s}=\HGA^{s}_{\XX/S}(r)
\end{equation}
which is a $R^1f_*(\mu_r)$-cover.  Thus we have:
\begin{prop}\label{prop_covering_morphism}
We get a covering morphism 
\begin{equation}\label{eqn_morphism_Ntw_HGA}
\Psi: \N^{\tw,s}_{\XX/S}(r,\sO)\to \HGA^{s}_{\XX/S}(r)
\end{equation}
which factors through $\left(\rHA_X(r)\right)^{\XX,s}$.
\end{prop}

\begin{rmk}
The stack $\left(\rHA_X(r)\right)^{\XX,s}$ is actually isomorphic to the moduli stack of twisted stable Higgs sheaves in sense of Yoshioka \cite{Yoshioka} which is studied in \cite[\S 3.3, \S 3.4]{Jiang_2019}.
\end{rmk}

\subsubsection{Perfect obstruction theory for twisted Higgs sheaves}

For the stack  $\N^{\tw}:=\N^{\tw,s}_{\XX/S}(r,\sO,c_2)$ of $\XX$-twisted stable Higgs sheaves on $\XX$, we recall the perfect obstruction theory in \cite[\S 4.2]{Jiang_2019}. 

We work on a family $\XX\to X\to  S$ over a scheme $S$. 
Recall that $\YY=\Tot(K_{\XX})$, and the category of $\XX$-twisted  Higgs sheaves $(E,\phi)$ on $\XX$ is equivalent to the category of torsion sheaves on $\YY$ supported on $\XX\subset \YY$. 
Let us pick up a twisted  universal sheaf $\rE$ over $\N^{\tw}\times_{S}\YY$.  
We use the same $\pi$ to represent the projection 
$$\pi: \YY\to \XX; \quad  \pi: \N^{\tw}\times_{S}\YY\to \N^{\tw}\times_{S}\XX.$$
We then let 
$$\E:=\pi_*\rE \text{~on~} \N^{\tw}\times_{S}\XX$$
is flat over $\N^{\tw}$.  $\E$ is also coherent because it can be seen locally on $\N^{\tw}$, and is also a $\XX$-twisted sheaf. It defines a classifying map:
$$\Pi: \N^{\tw}\to \sM^{\tw}$$
by
$$\sE\mapsto \pi_*\sE; \quad  (E,\phi)\mapsto E,$$
where $\sM^{\tw}$ is the moduli stack of $\XX$-twisted coherent sheaves on the fibre of $\XX\to S$. 
We use the same $\E$ over $\sM^{\tw}\times \XX$ and $\E=\Pi^*\E$ on $\N^{\tw}\times \XX$. 
Let 
$$p_{\YY}:  \N^{\tw}\times_{S}\YY\to \N^{\tw}; \quad   p_{\XX}:  \N^{\tw}\times_{S}\XX\to \N^{\tw}$$
be the projections.  From \cite[(4.2.4)]{Jiang_2019}, we have:
\begin{equation}\label{eqn_deformation2}
R\cHom_{p_{\YY}}(\rE, \rE)\stackrel{\pi_*}{\longrightarrow}R\cHom_{p_{\XX}}(\E, \E)\stackrel{[\cdot, \phi]}{\longrightarrow}R\cHom_{p_{\XX}}(\E, \E\otimes K_{\XX}).
\end{equation}
Taking the relative Serre dual of the above exact triangle we get 
\begin{equation}\label{eqn_deformation3}
R\cHom_{p_{\XX}}(\E, \E)[2]\to R\cHom_{p_{\XX}}(\E, \E\otimes K_{\XX/B})[2]\to R\cHom_{p_{\YY}}(\rE, \rE)[3].
\end{equation}
From \cite[Proposition 4.1]{Jiang_2019}, (\ref{eqn_deformation3})
is the same as (\ref{eqn_deformation2}), just shifted.

We borrow the following commutative diagram in \cite[Corollary 2.22]{TT1}, in which 
 the exact triangle  (\ref{eqn_deformation2}) fits:
\[
\xymatrix{
R\cHom_{p_{\XX}}(\E, \E\otimes K_{\XX/S})_{0}[-1]\ar[r]\ar@{<->}[d] &R\cHom_{p_{\YY}}(\rE, \rE)_{\perp}
\ar[r]\ar@{<->}[d] & R\cHom_{p_{\XX}}(\E, \E)_{0}\ar@{<->}[d]\\
R\cHom_{p_{\XX}}(\E, \E\otimes K_{\XX/S})[-1]\ar[r]\ar@{<->}[d]_{\id}^{\tr} &R\cHom_{p_{\YY}}(\rE, \rE)
\ar[r]\ar@{<->}[d] & R\cHom_{p_{\XX}}(\E, \E)\ar@{<->}[d]_{\id}^{\tr}\\
Rp_{\XX *}K_{\XX/S}[-1]\ar@{<->}[r]& Rp_{\XX *}K_{\XX/S}[-1]\oplus Rp_{\XX *}\sO_{\XX}\ar@{<->}[r]& Rp_{\XX *}\sO_{\XX}
}
\]
where $(-)_0$ denotes the trace-free Homs.  The $R\cHom_{p_{\YY}}(\rE, \rE)_{\perp}$ is the co-cone of the middle column. 
The truncation $\tau^{[-1,0]}R\cHom_{p_{\YY}}(\rE, \rE)$ defines a symmetric perfect obstruction theory on the moduli space $\N^{\tw}$. 

Now consider the morphism (\ref{eqn_morphism_Ntw_HGA}) in Proposition \ref{prop_covering_morphism}:
$$\Psi: \N^{\tw}\to \HGA^{s}_{\XX/S}(r,c_2),
$$
where $\HGA^{s}_{\XX/S}(r,c_2)$ represents the stack of generalized Higgs-Azumaya algebras on $X$ with cohomology class 
$[\XX]$, and of rank $r^2$ and second Chern class $c_2$.  
Let $\HA:=(\mathbf{A}, \Phi)$ be the universal generalized Higgs-Azumaya algebra on 
$X\times \HGA^s_{\XX/S}(r)$ whose fiber over the moduli space have cohomology class 
$[\XX]$. Let 
$$\mathbf{p}: \XX(\HA)\to X\times \HGA^s_{\XX/S}(r)$$
be the gerbe of trivialized trivializations of $\HA$.  Then we have 
$$\mathbf{A}\cong Rp_* R\sE nd(\widetilde{\E})$$
for an $\XX(\HA)$-twisted sheaf $\widetilde{\E}$, and a universal Higgs filed 
$\Phi$. Thus $(\widetilde{\E}, \Phi)$ defines a universal $\XX(\HA)$-twisted sheaf $\widetilde{\rE}$
on $\YY\times \N^{\tw}$.  Moreover, the covering morphism $\Psi$ gives rise to an isomorphism: 
$$\widetilde{p}: 
\XX(\HA)\times_{X\times \HGA^s_{\XX/S}(r)}X\times \N^{\tw}\stackrel{\sim}{\longrightarrow}
\XX\times \N^{\tw}$$
and $\widetilde{p}^*\E$ is a universal sheaf $\widetilde{\E}$. 
Also we have an isomorphism: 
$$\widetilde{p}: 
\XX(\HA)\times_{Y\times \HGA^s_{\XX/S}(r)}Y\times \N^{\tw}\stackrel{\sim}{\longrightarrow}
\YY\times \N^{\tw}$$
and $\widetilde{p}^*\rE$ is a universal sheaf $\widetilde{\rE}$.  Thus we get 
\begin{equation}\label{eqn_RHom}
R\cHom(\rE, \rE)\cong \mathbf{L}\widetilde{p}^*\HA.
\end{equation}
Now let $\HA_0\subset \HA$ be the trace-less part and we have
$$
R\cHom(\rE, \rE)_0^{\perp}\cong \mathbf{L}\widetilde{p}^*\HA_0.
$$
So from \cite[Proposition 3.4.3.7]{Lieblich_ANT}, or the virtual pullback we have:

\begin{thm}\label{thm_POT_HGA}
Let $L^{\bullet}_{\HGA}$ be the cotangent complex for the stack $\HGA^s_{\XX/S}(r)$. Then 
$$\HA_0\to L^{\bullet}_{\HGA}$$
defines a symmetric obstruction theory on $\HGA^s_{\XX/S}(r)$ in sense of Behrend \cite{Behrend}.
\end{thm}
\begin{proof}
This is from the compatible diagram:
\[
\xymatrix{
R\cHom(\rE, \rE)_0^{\perp}\ar[r]\ar[d]_{\cong}& L^{\bullet}_{\N^{\tw}}\ar[d]^{\cong}\\
\mathbf{L}\widetilde{p}^*\HA_0\ar[r]& \Psi^*L^{\bullet}_{\HGA}
}
\]
The symmetric property of the perfect obstruction theory  $\HA_0\to L^{\bullet}_{\HGA}$ is just from the symmetric obstruction theory of $R\cHom(\rE, \rE)_0^{\perp}$.
\end{proof}

\subsection{Application to Vafa-Witten invariants}

Recall that in \cite[\S 4.3]{Jiang_2019}, let 
$\N^{\perp,\tw}_{\sO}\subset \N^{\tw}$ be the the moduli stack of $\XX$-twisted rank $r$ stable Higgs sheaves  with 
$\det(E)=\sO$ and second Chern class $c_2$. Then  
$$R\cHom_{p_{\YY}}(\rE, \rE)_{\perp}[1]\Tt^{-1}\longrightarrow \ll_{\N^{\perp,\tw}_{\sO}}$$
defines a symmetric obstruction theory. 

Here is the definition of twisted Vafa-Witten invariants in \cite[Definition 4.4]{Jiang_2019}.
\begin{defn}\label{defn_SU_twisted_VW_invariants}
Let $\N^{\perp,\tw}_{\sO}$ be the moduli space of $\XX$-twisted stable Higgs sheaves.  Then the  twisted Vafa-Witten invariants are defined as:
$$\VW^{\tw}(\XX):=\int_{[(\N^{\perp,\tw}_{\sO})^{\Gm}]^{\vir}}\frac{1}{e(N^{\vir})}.$$
Here $\N^{\perp,\tw}_{\sO}$ admits a $\Gm$-action induced by the $\Gm$-action on $K_{\XX}$ by scaling the fibers and 
$[(\N^{\perp,\tw}_{\sO})^{\Gm}]^{\vir}$ is the virtual fundamental class on the fixed loci and $e(N^{\vir})$ is the Euler class of the virtual normal bundle. The invariants are defined by the virtual localization in \cite{GP}.  This corresponds to the gauge group 
$\SU(r)/\zz_r$. 
\end{defn}

The symmetric obstruction theory $\HA_0$ on $\HGA^s_{\XX/\kappa}(r)$ also defines a virtual fundamental class 
$$[\HGA^s_{\XX/\kappa}(r)]^{\vir}\in H_0(\HGA^s_{\XX/\kappa}(r)).$$
The stack $\HGA^s_{\XX/\kappa}(r)$ also admits a $\Gm$-action given by 
$\lambda(A,\Phi)=(A,\lambda \Phi)$ by scaling the Higgs fields.  We write 
$$\left(\HGA^s_{\XX/\kappa}(r)\right)^{\Gm}$$
as the $\Gm$-fixed loci.

We discuss the $\Gm$-fixed loci  $\left(\N^{\perp,\tw}_{\sO}\right)^{\Gm}$ and $\left(\HGA^s_{\XX/\kappa}(r)\right)^{\Gm}$. 

\subsubsection{Case I-Instanton Branch:}\label{subsubsec_first_type} 
This component corresponds to $\Gm$-fixed
 Higgs pairs $(E,\phi)$ such that $\phi=0$. Then twisted sheaf $E$ must be stable, and this component is the moduli space 
 $\sM_{\XX/\kappa}^{s,\tw}:=\sM^{s,\tw}_{\XX/\kappa}(r,\sO,c_2)$ of 
$\XX$-twisted stable sheaves on $\XX$ with rank $r$, fixed determinant $\sO$ and second Chern class $c_2$.  The exact triangle in (\ref{eqn_deformation2}) splits the obstruction theory
$$R\cHom_{p_{\YY}}(\rE, \rE)_{\perp}[1]\Tt^{-1}\cong R\cHom_{p_{\XX}}(\E, \E\otimes K_{\XX})_{0}[1]\oplus 
R\cHom_{p_{\XX}}(\E, \E)_{0}[2]\Tt^{-1}$$
where $\Tt^{-1}$ represents the moving part of the $\Gm$-action. Then the $\Gm$-action induces a perfect obstruction theory 
$$E_{\sM}^{\bullet}:=R\cHom_{p_{\XX}}(\E, \E\otimes K_{\XX})_{0}[1]\to \ll_{\sM_{\XX/\kappa}^{s,\tw}}.$$
The virtual normal bundle 
$$N^{\vir}=R\cHom_{p_{\XX}}(\E, \E\otimes K_{\XX})_{0}\Tt=E_{\sM}^{\bullet}\otimes \Tt[-1].$$
The twisted Vafa-Witten  invariant contributed from $\sM_{\XX/\kappa}^{s,\tw}$ is 
$$
\int_{[\sM_{\XX/\kappa}^{s,\tw}]^{\vir}}\frac{1}{e(N^{\vir})}
=\int_{[\sM_{\XX/\kappa}^{s,\tw}]^{\vir}}c_{\vd}(E_{\sM}^{\bullet})\in \zz.
$$
This is the signed virtual Euler number of Ciocan-Fontanine-Kapranov/Fantechi-G\"ottsche, see \cite{FG}. 

On the stack  $\left(\HGA^s_{\XX/\kappa}(r)\right)^{\Gm}$, the first component corresponds to the $\Gm$-fixed
 Higgs-Azumaya pairs $(A,\Phi)$ such that $\Phi=0$.  It is the stack $\GA^s_{\XX/\kappa}(r)$ of stable generalized Azumaya algebras $A$ on $X$ with cohomology class $[\XX]$.

\subsubsection{Case II-Monopole Branch:}\label{subsec_second_fixed_loci}
The second component $\sM^{(2)}$ corresponds to the case that in the $\Gm$-fixed $\SS$-twisted stable Higgs sheaf $(E,\phi)$, the Higgs field $\phi\neq 0$.  This component corresponds to nested Hilbert schemes, see \cite{Jiang_2019}.
For  $\left(\HGA^s_{\XX/\kappa}(r)\right)^{\Gm}$, the second component $\sM^{(2)}$ corresponds to  $\Gm$-fixed
 Higgs-Azumaya pairs $(A,\Phi)$ such that $\Phi\neq 0$.

\subsubsection{Relationship between virtual classes}

We present the pullback properties of virtual classes:

\begin{thm}\label{thm_pullback_virtual_class}
The morphism 
$$\Psi: \N^{\tw}\to \HGA^s_{\XX/\kappa}(r)$$
and the obstruction theories $E_{\N^{\tw}}^{\bullet}$ and $\HA_0$
are compatible with the $\Gm$-action, therefore gives 
$$\Psi^{!}\left([\HGA^s_{\XX/\kappa}(r)]^{\vir}\right)=[\N^{\perp, \tw}_{\sO}]^{\vir}.$$
We also have:
$$\Psi^{!}\left(\Big[\left(\HGA^s_{\XX/\kappa}(r)\right)^{\Gm}\Big]^{\vir}\right)=\Big[\left(\N^{\perp, \tw}_{\sO}\right)^{\Gm}\Big]^{\vir}$$
and 
$$\Psi^{!}\left(N_{\N^{\tw}}^{\vir}\right)=N_{\HGA}^{\vir},$$
where $N_{\N^{\tw}}^{\vir}$ and $N_{\HGA}^{\vir}$ are the virtual normal bundles for the obstruction theories $E_{\N^{\tw}}^{\bullet}$ and $\HA_0$.  Thus the $\PGL_r$-invariant
$$\VW_{\PGL}(X)=\int_{\Big[\left(\HGA^s_{\XX/\kappa}(r)\right)^{\Gm}\Big]^{\vir}}
\frac{1}{e(N_{\HGA}^{\vir})}$$
can be calculated by the twisted invariant $\VW^{\tw}(\XX)$ in Definition \ref{defn_SU_twisted_VW_invariants}.
\end{thm}
\begin{proof}
This is the perfect obstruction theory of $\HA_0$ in Theorem \ref{thm_POT_HGA} and isomorphism (\ref{eqn_RHom}). 
Also since the $\Gm$-action is given by scaling the Higgs fields $\phi$ and $\Phi$, it is compatible with the symmetric obstruction theories.  Thus the results are from the virtual pullback property in \cite{Manolache}. 
\end{proof}

The following corollary is from the instanton branch:
\begin{cor}
When restricting to the instanton branch, the perfect obstruction theory 
$$E_{\sM}^{\bullet}:=R\cHom_{p_{\XX}}(\E, \E\otimes K_{\XX})_{0}[1]\to \ll_{\sM_{\XX/\kappa}^{s,\tw}}$$
is the same as the pullback of the perfect obstruction theory 
$\mathbf{A}_0\to L^{\bullet}_{\GA^s_{\XX/\kappa}}$ in \cite[\S 6.5.2]{Lieblich_ANT}. 
\end{cor}



\subsection*{}


\begin{thebibliography}{12}  
\bibitem{ACV}D. Abramovich, A. Corti and A. Vistoli, \newblock Twisted bundles and admissible covers, {\em Comm. Algebra}, 31 (8):3547-3618, 2003. Special issue in honor of Steven L. Kleiman. 

\bibitem{Behrend}  K. Behrend, \newblock Donaldson-Thomas invariants via microlocal geometry, 
{\em Ann. Math.} (2009), Vol. 170, No.3, 1307-1338, math.AG/0507523.
\bibitem{BF}   K. Behrend and B. Fantechi, \newblock The intrinsic normal cone,
                   alg-geom/9601010, {\em Invent. Math}. 128 (1997), no. 1, 45-88.
    
\bibitem{de_Jong}A. J. de Jong, \newblock A result of Gabber, 2003, Preprint.              
\bibitem{FG}B. Fantechi and L. Gottsche, \newblock  Riemann-Roch theorems and elliptic genus for virtually smooth schemes, {\em Geom.  Top.} 14 83-115, 2010. arXiv:0706.0988.
\bibitem{GWZ}M. Groechenig, D. Wyss and P. Ziegler, \newblock. Mirror symmetry for moduli spaces of Higgs bundles via p-adic integration,  arXiv:1707.06417.

\bibitem{GP}T.~Graber and R.~Pandharipande, \newblock Localization of virtual classes, {\em Invent.~Math.}~\textbf{135} 487-518, 1999. alg-geom/9708001.
\bibitem{HT}T. Hausel and M. Thaddeus, \newblock  Mirror symmetry, Langlands duality, and the Hitchin system, {\em Invent. Math.}, 153 (2003), no. 1, 197-229.

\bibitem{JT}Y. Jiang and R. Thomas, \newblock Virtual signed Euler characteristics, {\em Journal of Algebraic Geometry}, 26 (2017) 379-397, arXiv:1408.2541.
\bibitem{Jiang}Y. Jiang, \newblock Note on MacPherson's local Euler obstruction, {\em Michigan Mathematical Journal}, 68 (2019), 227-250, arXiv:1412.3720.

\bibitem{JP}Y. Jiang, and P. Kundu, \newblock  The Tanaka-Thomas's Vafa-Witten invariants for surface Deligne-Mumford stacks, arXiv:1903.11477.

\bibitem{Jiang_ICCM}Y. Jiang, \newblock  The Vafa-Witten invariants for surface Deligne-Mumford stacks and S-duality, Survey for ICCM-2019, arXiv:1909.03067.
\bibitem{Jiang_2019}Y. Jiang, \newblock  Counting twisted sheaves and S-duality, preprint, arXiv:1909.04241.

\bibitem{LMB}G. Laumon and L. Moret-Bailly, \newblock  {\em Champs Algebriques},   Ergebnisse der Mathematik und ihrer Grenzgebiete. 3. Folge / A Series of Modern Surveys in Mathematics.

\bibitem{LT} J. Li and G. Tian, \newblock  Virtual moduli cycles and Gromov-Witten 
invariants of algebraic varieties, {\em J. Amer. Math. Soc.}, 11, 119-174, 1998, math.AG/9602007. 

\bibitem{Lieblich_Duke}M . Lieblich, \newblock Moduli of twisted sheaves, {\em Duke Math. J.}, 
Vol. 138, No. 1 (2007), 23-118.
\bibitem{Lieblich_Compo}M . Lieblich, \newblock Twisted sheaves and the period-index problem, {\em Compositio. Math.},  Vol. 144, Iss. 1 (2008) , 1-31.
\bibitem{Lieblich_ANT}M . Lieblich, \newblock Compactified moduli of projective bundles, {\em Algebra Number Theory},   Vol. 3 (2009), No. 6, 653-695.
\bibitem{Manolache}C. Manolache, \newblock Virtual pull-backs,   {\em Journal Algebraic Geometry}, Vol:21,(2012), 201-245.	arXiv:0805.2065. 
   
\bibitem{Milne}J. S. Milne, \newblock \'Etale cohomology,  Princeton University Press, Princeton, New Jersey. 
\bibitem{Stack_Project} Stack Project: https://stacks.math.columbia.edu/download/stacks-morphisms.pdf.

\bibitem{TT1} Y. Tanaka and R. P. Thomas, \newblock Vafa-Witten invariants for projective surfaces I: stable case, 
{\em J. Alg. Geom.} to appear, 
arXiv.1702.08487.
\bibitem{TT2}Y. Tanaka and R. P. Thomas, \newblock Vafa-Witten invariants for projective surfaces II: semistable case, {\em Pure Appl. Math. Q.} 13  (2017), 517-562, Special Issue in Honor of Simon Donaldson, 
arXiv.1702.08488.
\bibitem{VW} C.~Vafa and E.~Witten, \emph{A strong coupling test of S-duality}, Nucl. Phys. \textbf{B\,431} 3--77, 1994. hep-th/9408074.
\bibitem{Vistoli}A. Vistoli, \newblock Intersection theory on algebraic stacks and on their moduli spaces, {\em Invent. Math.}, 97
(1989) 613-670.
\bibitem{Yoshioka}K. Yoshioka, \newblock Moduli spaces of twisted sheaves on a projective variety,  Advanced Studies in Pure Mathematics 45, 2006 Moduli Spaces and Arithmetic Geometry (Kyoto, 2004) pp. 1-42. 
\bibitem{Yun}Z. Yun, \newblock Global Springer theory, {\em  Advances in Mathematics}, 
Vol. 228, Iss. 1, (2011),  266-328. 
\end{thebibliography}
\end{document}